\documentclass[reqno]{amsart}
\usepackage{amsmath}
\usepackage{amsfonts}
\usepackage{amssymb}

\newcommand\R{{\mathbb{R}}}
\newcommand\C{{\mathbb{C}}}

\theoremstyle{plain}
  \newtheorem{theorem}[subsection]{Theorem}

  \newtheorem{lemma}[subsection]{Lemma}

\theoremstyle{remark}
  \newtheorem{remark}[subsection]{Remark}

\theoremstyle{definition}

\begin{document}

\title{On extremisers to a bilinear Strichartz inequality}

\author{Shuanglin Shao}
\address{Department of Mathematics, University of Kansas, Lawrence, KS 66045}
\email{slshao@ku.edu}
\thanks{The author was supported in part by NSF Grant \# 1160981.}

%    General info

%\date{January 1, 2001 and, in revised form, June 22, 2001.}


\begin{abstract}
In this paper, we show that a pair of Gaussian functions are extremisers to a bilinear Strichartz inequality on $\mathbb{R}\times \mathbb{R}^2$, and unique up to the symmetry group of the inequality. 
\end{abstract}

\maketitle

\section{Introduction}

We consider the free Schr\"odinger equation
\begin{equation}\label{eq:free-schrodinger}
i\partial_t u+\Delta u=0,
\end{equation}
with initial data $u(0,x)=f(x)$ where $u:\mathbb{R}\times \mathbb{R}^d\to \mathbb{C}$ is
a complex-valued function and $d\ge 1$. We denote the solution $u$ by
using the Schr\"odinger evolution operator $e^{it\Delta}$:
\begin{equation}\label{eq-schrodinger}
u(t,x):=e^{it\Delta}f(x):= \frac {1}{(2\pi)^d}\int_{\mathbb{R}^d}
e^{ix\cdot\xi+it|\xi|^2}\widehat{f}(\xi)d\xi,
\end{equation}
where $\widehat{f}$ is the spatial Fourier transform of $f$ defined via
\begin{equation}\label{eq:def-fourier-transfm}
\widehat{f}(\xi):=\int_{\mathbb{R}^d}e^{-ix\cdot\xi}f(x)dx,
\end{equation}
where $x\cdot\xi$ (abbr. $x\xi$) denotes the Euclidean inner
product of $x$ and $\xi$ in the spatial space $\mathbb{R}^d$. When $f\in L^2(\mathbb{R}^d)$, the solution $e^{it\Delta} f$ enjoys a space-time estimate, the Strichartz inequality, 
\begin{equation}\label{eq-strichartz}
\|e^{it\Delta}f\|_{L^{2+\frac 4d}_{t,x}(\mathbb{R}\times \mathbb{R}^d)}
\le C\|f\|_{L^2_x(\mathbb{R}^d)},
\end{equation}
for some $C >0$, see e.g. \cite{Ginibre-Velo:1992:non-endpoint-Strichartz-inequality} or \cite{Tao:2006-CBMS-book}. Define 
$$C_d := \sup \{ \frac {\|e^{it\Delta}f\|_{L^{2+\frac 4d}_{t,x}(\mathbb{R}\times \mathbb{R}^d)} }{\|f\|_{L^2_x(\mathbb{R}^d)}}: \, f\in L^2(\mathbb{R}^d) \, f\neq 0\}.$$
Several authors have investigated the extremal problem for \eqref{eq-strichartz}, which asks whether there is an extremal function $f\in L^2(\mathbb{R}^d)$ such that $$\|e^{it\Delta}f\|_{L^{2+\frac 4d}_{t,x}(\mathbb{R}\times \mathbb{R}^d)} = C_d \|f\|_{L^2(\mathbb{R}^d)}, $$ and what properties the extremal functions have.  More precisely, Kunze \cite{Kunze:2003:maxi-strichartz-1d} treated the $d=1$ case and showed that extremisers exist by a longer concentration-compactness argument; when $d=1,2$, Foschi \cite{Foschi:2007:maxi-strichartz-2d} explicitly determined the
best constants and showed that the extremisers are Gaussians, and they are unique up to the symmetry of the Strichartz inequality. Hundertmark and Zharnitsky
\cite{Hundertmark-Zharnitsky:2006:maximizers-Strichartz-low-dimensions}
independently obtained this result. Carneiro \cite{Carneiro:2008:sharp-strichartz-norm} consider similar extremal questions for some Strichartz-type inequalities.  In \cite{Bennett-Bez-Carbery-Hundertmark:2008:heat-flow-of-strichartz-norm}, by using the method of heat-flow, Bennett, Bez, Carbery and Hundertmark offered a new proof to determine the best constants and the explicit form of extremisers for the symmetric inequality \eqref{eq-strichartz} when $d=1,2$. When $d\ge 3$, only the existence of extermisers for the symmetric Strichartz inequality has been shown, see e.g. \cite{Shao:2008:maximizers-Strichartz-Sobolev-Strichartz}. Similar extremal questions have been treated for the Fourier restriction inequality for the hyper-surfaces in the Euclidean spaces such as the sphere in \cite{Christ-Shao:extremal-for-sphere-restriction-I-existence, Christ-Shao:extremal-for-sphere-restriction-II-characterizations}, and the Strichartz inequality for the wave equation in \cite{Bulut:2009:maximizer-wave, Foschi:2007:maxi-strichartz-2d}.

In this paper, we specify the dimension $d=2$ and consider the extremal problem for a bilinear Strichartz inequality for the Schr\"odinger operator,
\begin{equation}\label{bilinear-strichartz}
\|e^{it\Delta} f e^{it\Delta} g\|_{L^2_{t,x}(\R\times \R^2)} \le \textbf{B} \|f\|_{L^2(\R^2)} \|g\|_{L^2(\R^2)}.
\end{equation}
where $\textbf{B}$ is the optimal constant defined by
\begin{equation}\label{eq-40}
\textbf{B}:=\sup_{f\neq 0,g\neq 0} \frac {\|e^{it\Delta} f e^{it\Delta} g\|_{L^2_{t,x}(\R\times \R^2)}}
{\|f\|_{L^2(\R^2)} \|g\|_{L^2(\R^2)}}.
\end{equation}
We define an extremiser or an extremal function to \eqref{bilinear-strichartz} is a pair of nonzero functions $(f,g) \in L^2\times L^2$ such that
$$\|e^{it\Delta} f e^{it\Delta} g\|_{L^2_{t,x}(\R\times \R^2)} = \textbf{B} \|f\|_{L^2(\R^2)} \|g\|_{L^2(\R^2)}.$$

It is well known that the linear Strichartz inequality \eqref{eq-strichartz} is invariant under the following symmetry group $G$ generated by
\begin{itemize}
\item Translation. $e^{it\Delta}f(x) \to e^{i(t-t_0)\Delta}f(x-x_0)$ for any $(t_0,x_0)\in \R\times \R^2$.
\item Scaling. $e^{it\Delta}f(x) \to \lambda^2 e^{i\lambda^2 t\Delta}f( \lambda x) $ for any $\lambda >0$.
\item Galilean transform. $e^{it\Delta}f(x) \to e^{ix\cdot \xi_0+it|\xi_0|^2} f(x+2t\xi_0) $ for any $\xi_0\in \R^2$.
\item Phase transition. $e^{it\Delta}f(x) \to \alpha e^{it\Delta}f(x)$ for $\alpha\in \mathbb{C}\setminus \{0\}$.
\item Space rotation. $e^{it\Delta}f(x) \to e^{it\Delta}f(Rx)$ for any $R\in SO(2)$.
\end{itemize}
The bilinear Strichartz inequality is invariant when the same symmetry group acts simultaneously  on $(f,g)$ in \eqref{bilinear-strichartz}. It is additionally allowed the following phase transition
$$ e^{it\Delta}f(x) \to \alpha e^{it\Delta}f(x), e^{it\Delta}g(x) \to \beta e^{it\Delta}g(x)$$
for $\alpha \neq 0 $ and $\beta \neq 0$.  For the extremal problem for the linear Strichartz inequality, it is true that the symmetry group $G$ changes an extremal function to another; so an extremal function $f$ will generate a family of extremal functions under the action of $G$. This family of functions is called the orbit of $f$.  Because the inequality \eqref{bilinear-strichartz} is invariant under the symmetry in $G$, it is also the case for the extremal function. Now we state the following result.

\begin{theorem}\label{thm}
The pair of Gaussian functions $$(f,g) = \bigl(\exp{(-|x|^2)}, \exp{(-|x|^2)} \bigr)$$ is an extremiser to the bilinear Strichartz inequality \eqref{bilinear-strichartz}, and $\textbf{B}=\frac 12$. Moreover, the set of extremisers for which \eqref{bilinear-strichartz} holds coincides with the orbit of
$$ (f,g) = \bigl(\exp{(A|x|^2+b\cdot x +C_1)}, \exp{(A|x|^2+b\cdot x +C_2)} \bigr),$$
where $A \in \mathbb{C}$ with the real part $\Re(A) <0$ and $ C_1, C_2 \in \mathbb{C}$ and $b\in \mathbb{C}^2$.
\end{theorem}

To establish Theorem \ref{thm}, we follow Foschi's argument in \cite{Foschi:2007:maxi-strichartz-2d}, where it is shown that Guassians are the only extremisers to the linear Strichartz inequality
\begin{equation}\label{linear-strichartz}
\|e^{it\Delta}{f}\|_{L^4_{t,x} (\R \times \R^2)} \le C_2\|f\|_{L^2(\R^2)}
\end{equation} up to the symmetry in $G$.

\begin{remark}
An analogous theorem to Theorem \ref{thm} can be established for a trilinear Strichartz inequality in the one-dimensional case,
\begin{equation}\label{eq-10}
\|e^{it\Delta} f e^{it\Delta} g e^{it\Delta} h \|_{L^2_{t,x}(\mathbb{R} \times \mathbb{R})}\le C \|f\|_{L^2} \|g\|_{L^2} \|h\|_{L^2}.
\end{equation}
Roughly speaking, $(f,g,h)$ is a triple of extremisers if and only if they are scalar multiples of a Gaussian function.
\end{remark}

\begin{remark}
Recently it comes to our attention that M. Charalambides \cite{Charalambides:2012:Cauchy-Pexider} systematically investigated the question of characterizing functions $f, g$ and $h$ such that the Cauchy-Pexider functional equation $f(x)g(y)=h(x+y)$ with $x,y$ on some hyper-surface in $R^{d+1}$. The solutions to such functional equation are uniquely determined to be exponential affine functions. This is closely connected to Theorem \ref{thm} because the functional equation characterizing the sharpness of the bilinear Strichartz inequality \eqref{bilinear-strichartz} is in the same form, see Section \ref{proof}. 
\end{remark}

{\bf Acknowledgment.} The author is grateful to B. Pausader,  and J. Jiang for many helpful discussions.

\section{Notation and preliminary}

We begin with some notation in $\mathbb{R}^d$. Define the Fourier transform,
$$\mathcal{F}(f) (\xi)= \widehat{f}(\xi) =\int_{\R^d} e^{-ix\cdot \xi} f(x)dx, \quad \xi\in \R^d.$$
The inverse of the Fourier transform,
$$\mathcal{F}^{-1} (\widehat{f})(x)=f(x) =\frac {1}{(2\pi)^d}\int_{\R^d} e^{ix\cdot \xi} \widehat{f}(\xi)d\xi.$$
We also use these notations to indicate the Fourier transform or the inverse Fourier transform in space-time $(t,x)$. 

The Plancherel theorem states that
$$ \|f\|_{L^2(\R^d)}= \frac {1}{(2\pi)^{d/2}} \|\widehat{f}\|_{L^2(\R^d)}. $$
Moreover the Parseval identity states that
$$ \int_{\R^d}f(x)\overline{g}(x)dx =\frac {1}{(2\pi)^d}\int_{\R^d} \widehat{f}(\xi)\overline{\widehat{g}}(\xi)d\xi.$$

Let $\sigma$ be the endowed measure on the paraboloid $P:=\{(\tau,\xi): \tau =|\xi|^2\}$ in $\R^3$, defined to be the pullback of the Lebesgue measure under the projection map: $(|\xi|^2,\xi)\mapsto \xi$. We write 
$$ e^{it\Delta}f(x)=\frac {1}{(2\pi)^d}\int_{\R^d} e^{ix\cdot\xi+it|\xi|^2} \widehat{f}(\xi) d\xi. $$
We lift $\widehat{f}$ onto the paraboloid $P:= \{(\tau,\xi): \, \tau= |\xi|^2\}$ via 
\begin{equation}\label{eq-45}
F(\xi, |\xi|^2) = \widehat{f}(\xi).
\end{equation}
Then by using the notation of the adjoint Fourier restriction operator for the paraboloid \cite[p. 347]{Stein:1993},
\begin{equation}\label{eq-46}
e^{it\Delta}f(x) =2\pi \mathcal{F}^{-1}(F\sigma) (t,x),
\end{equation}
Here $\mathcal{F}^{-1}$ is understood as the inverse Fourier transform on $\R^{d+1}$; $F\sigma = \hat{f}(\xi) \delta(\tau-|\xi|^2)$ is the measure supported on the paraboloid in $\R^{d+1}$.  

We define the convolution of $f$ and $g$,
$$f*g (x) =\int_{\R^d} f(x-y)g(y)dy,$$
and record a useful identity about convolution under the action of Fourier transform
\begin{equation}\label{eq-1}
 \widehat{f* g} = \widehat{f}\widehat{g}
 \end{equation}
 For $f,g$ on $\R^{d}$,
 \begin{equation} 
 \mathcal{F}\bigl(fg\bigr)= \frac {1}{(2\pi)^{d}} \mathcal{F}{f}* \mathcal{F}{g}.
 \end{equation}

We also record several lemmas from Foschi \cite{Foschi:2007:maxi-strichartz-2d}. The following lemma shows that the convolution of the surface measure $\sigma$ is constant on its corresponding support $\Omega:=\{(\tau, \xi)\in \R\times \R^2: \, 2\tau \ge |\xi|^2\}$, see e.g. \cite[Lemma 3.2]{Foschi:2007:maxi-strichartz-2d} .
 \begin{lemma}\label{le-convolution-constant}
 \begin{equation}\label{eq-9}
 \text{If } x\in \Omega, \quad \sigma * \sigma (x) =\pi/2.
 \end{equation}
 \end{lemma}
The proof uses that the convolution of measure supported on paraboloid is invariant under the mapping $(x',x_3) \to (0, x_3 - \frac {|x'|^2}{2})$, and the dilation symmetry of the paraboloid.

Finally we cite a lemma \cite[Proposition 7.15]{Foschi:2007:maxi-strichartz-2d} on characterizing the following functional inequality,
\begin{equation}\label{eq-functional-equation}
f(x) f(y) = H(|x|^2+|y|^2, x+y), \text{ for almost everywhere } (x,y) \in \R^2 \times \R^2.
\end{equation}
\begin{lemma}\label{le-functional-equation}
If $f: \R^2 \to \C$ and $H: \Omega \to \C $ are nontrivial locally integrable functions which satisfy the functional equation \eqref{eq-functional-equation}, then there exists constants $A \in \C$, $b\in \C^2$ and $C\in \C$ such that
$$ f(x) =\exp(A |x|^2 + b\cdot x+ C), \quad H(t,x) =\exp (At + b\cdot x +2C) $$
for almost all $(t,x)\in \Omega$.
\end{lemma}

\section{The proof}\label{proof}
Before we prove Theorem \ref{thm}, we recall Foschi's argument in \cite{Foschi:2007:maxi-strichartz-2d}. Foschi establishes the inequality \eqref{linear-strichartz} with an explicit constant by the Cauchy-Schwarz inequality. The only place where an inequality sign occurs is due to the Cauchy-Schwarz inequality. Then the question  reduces to what functions make the Cauchy-Schwarz inequality sharp in the sense that the inequality becomes equal. This yields a functional equation, whose solutions uniquely determine extremisers. We will apply this idea.

\begin{proof}[Proof of Theorem \ref{thm}]
By the Plancherel theorem,
\begin{align*}
\|e^{it\Delta} f e^{it\Delta} g\|_{L^2(\R^3)} &=(2\pi)^2\|\mathcal{F}^{-1}(F\sigma)\mathcal{F}^{-1} (G\sigma)\|_{L^2(\R^3)} \\
&=(2\pi)^{2} (2\pi)^{-3/2}\|\mathcal{F}\left(\mathcal{F}^{-1}(F\sigma)\mathcal{F}^{-1} (G\sigma \right)\|_{L^2(\R^3)} \\
& =(2\pi)^{1/2}(2\pi)^{-3}\|{F\sigma}*{G\sigma}\|_{L^2(\R^3)} \\
&=\frac {1}{(2\pi)^{5/2}}
\|{F\sigma}*{G\sigma}\|_{L^2(\R^3)} .
\end{align*}
where $F(\xi, |\xi|^2) = \widehat{f}(\xi)$, $G(\xi, |\xi|^2) = \widehat{g}(\xi)$, respectively.

We write $x=(x',x_3)\in \R^2 \times \R$. Since the measure $\sigma$ on paraboloid $P$ is defined to be the pull-back of the Lebesgue measure,
\begin{equation*}
F\sigma = F(x', |x'|^2)dx'= F(x',x_3)\delta(x_3- |x'|^2), \, G\sigma = G(y',|y'|^2)dy'= G(y',y_3)\delta(y_3-|y'|^2).
\end{equation*} 
Then
\begin{align*}
&F\sigma * G\sigma (x) \\
&= \int\int \delta(x=y+z) F(y)\delta(y_3- |y'|^2) G(z)\delta(z_3-|z'|^2) dydz \\
&= \int\int \delta \left(\substack{ x'=y'+z' \\ x_3= |y'|^2+|z'|^2}\right)  F(y',|y'|^2) G(z',|z'|^2) dy'dz'
 \end{align*}
Denote by ${\bf 1}(x)$ the identity function. By using the Cauchy-Schwarz inequality, 
\begin{equation}\label{eq-49}
\begin{split}
&F\sigma * G\sigma (x) \\
 &=\int\int \left(F(y',|y'|^2) G(z',|z'|^2)\right) \left({\bf 1}(y',|y'|^2) {\bf 1}(z',|z'|^2)\right)\delta \left(\substack{ x'=y'+z' \\ x_3= |y'|^2+|z'|^2}\right) dy'dz' \\
&=\left(\int\int \left|F(y',|y'|^2) G(z',|z'|^2)\right|^2 \delta \left(\substack{ x'=y'+z' \\ x_3= |y'|^2+|z'|^2}\right) dy'dz'\right)^{1/2} \left(\int\int \left|{\bf 1}(y',|y'|^2) {\bf 1}(z',|z'|^2)\right|^2 \delta\left(\substack{ x'=y'+z' \\ x_3= |y'|^2+|z'|^2}\right) dy'dz'\right)^{1/2}\\
&= \left( |F|^2 \sigma * |G|^2 \sigma (x) \right)^{1/2} \left(\sigma *\sigma (x) \right)^{1/2}.  
\end{split}
\end{equation}
Therefore 
\begin{equation}\label{eq-4}
\begin{split}
\|F\sigma*G\sigma\|^2_{L^2(\R^3)} &=\int_{\R^3} \bigl| F\sigma*G\sigma(x)\bigr|^2dx\\
& \le \int_{\R^3} \bigl| |F|^2\sigma *|G|^2\sigma (x)\bigr| \bigl| \sigma*\sigma(x)\bigr| dx \\
&=\frac \pi 2 \int_{\R^3} |F|^2\sigma *|G|^2\sigma dx \\
&= \frac \pi 2 \|F\|_{L^2_\sigma}^2  \|G\|_{L^2_\sigma}^2,
\end{split}
\end{equation}
where we have used Lemma \ref{le-convolution-constant},
\begin{equation}\label{eq-3}
\sigma *\sigma (x)=\pi/2, \text{ for all } x\in \Omega.
\end{equation}
To conclude so far, we have established the bilinear Strichartz inequality \eqref{bilinear-strichartz} with an explicit constant $1/2$.

If the Cauchy-Schwarz inequality used in \eqref{eq-4} is sharp in the sense that an equal sign occurs, then all the inequalities in \eqref{eq-4} become equal; that is to say
$$ \|F\sigma*G\sigma\|_{L^2} =\sqrt{\frac \pi 2}  \|F\|_{L^2_\sigma} \|G\|_{L^2_\sigma}.$$

An examination of sharpness of the Cauchy-Schwarz inequality in \eqref{eq-49} shows that, there exists $\alpha \in \mathbb{C}$ such that\begin{equation}\label{eq-5}
F(x',|x'|^2) G(y',|y'|^2) = \alpha F(z',|z'|^2) G(w',|w'|^2) .
\end{equation}
if 
\begin{align}
\label{eq-51} |x'|^2 + |y'|^2 & = |z'|^2+ |w'|^2, \\
\label{eq-52}  x'+y' &=z'+w'. 
\end{align}
Thus by reducing $(F, G)$ to $(\hat{f}, \hat{g})$,
\begin{equation}\label{eq-6}
\hat{f}(x') \hat{g}(y')= H(|x'|^2+|y'|^2, x'+y') \text{  for a.e. } (x', y')\in \R^2 \times \R^2.
\end{equation} for some measurable function $H$. 

Since $\hat{f}$ is not identically $0$, without loss of generality, we may assume that $\hat{f}(0)\neq 0$. Let $x'=0$ and $y'=0$ in \eqref{eq-6}, respectively,
$$ \hat{f}(0) \hat{g}(x) = \hat{f}(x) \hat{g}(0) \text{  for a.e. } x\in \R^2.$$
Then
$$\hat{g}(x)=\frac {\hat{g}(0)}{\hat{f}(0)} \hat{f}(x).$$
In view of this, we may assume that $\hat{f}=\hat{g}$ up to constants. This reduces \eqref{eq-6} to
\begin{equation}\label{eq-8}
\hat{f}(x') \hat{f}(y')= H(|x'|^2+|y'|^2, x'+y') \text{  for a.e. } (x',y')\in \R^2 \times \R^2.
\end{equation}
Then by Lemma \ref{le-functional-equation},
there exists constants $A \in \C$, $b\in \C^2$, and $C\in \C$ such that
\begin{equation}\label{eq-7}
\hat{f}(x) =\exp{(A|x|^2+b\cdot x + C)}, \quad H(t,x) =\exp{(At+b\cdot x +2C)}
\end{equation} for almost everywhere $(t,x) \in P $. 
Since the Fourier transform of a Gaussian function is a Gaussian function, we have shown that, the extremisers to \eqref{bilinear-strichartz} are Gaussian functions, which are unique up to the symmetry specified in the first section. This completes the proof of Theorem \ref{thm}.
\end{proof}

\bibliographystyle{amsplain}

\begin{thebibliography}{10}

\bibitem {Bennett-Bez-Carbery-Hundertmark:2008:heat-flow-of-strichartz-norm} J.~Bennett, N.~Bez, A.~Carbery, and D.~Hundertmark,  \textit{Heat-flow monotonicity of {S}trichartz norms.} Analysis and PDE. \textbf{2} (2009), No. 2, 147--158.

\bibitem {Bulut:2009:maximizer-wave} A. ~Bulut, \textit{Maximizers for the {S}trichartz inequalities for the {W}ave Equation.} Differential and Integral Equations.
\textbf{23} (2010) 1035--1072.

\bibitem {Carneiro:2008:sharp-strichartz-norm} E.~Carneiro, \textit{A sharp inequality for the {S}trichartz norm.} Int. Math. Res. Not.
\textbf{16} (2009), 3127--3145.

\bibitem {Charalambides:2012:Cauchy-Pexider} M.~Charalambides, \textit{On Restricting Cauchy-Pexider Equations to Submanifolds.}  arXiv:1203.5370.


\bibitem {Christ-Shao:extremal-for-sphere-restriction-I-existence} M.~ Christ and S.~Shao, \textit{Existence of extremals for a {F}ourier restriction inequality.} Analysis and PDE. \textbf{5}(2), (2012), 261Ð312.

\bibitem {Christ-Shao:extremal-for-sphere-restriction-II-characterizations} M.~ Christ and S.~Shao, \textit{On the extremisers of an adjoint {F}ourier restriction inequality.} Advances in Mathematics.
\textbf{230}(2)  (2012), 957-977.

\bibitem {Foschi:2007:maxi-strichartz-2d} D.~Foschi, \textit{Maximizers for the {S}trichartz inequality.} J. Eur. Math. Soc. (JEMS)
\textbf{9} (4) (2007), 739--774.

\bibitem {Ginibre-Velo:1992:non-endpoint-Strichartz-inequality} J.~Ginibre  and G.~Velo, \textit{Smoothing properties and retarded estimates for some
              dispersive evolution equations.}Comm. Math.Phys.
\textbf{144} (1), (1992), 163--188.

\bibitem {Hundertmark-Zharnitsky:2006:maximizers-Strichartz-low-dimensions} D.~Hundertmark and V.~Zharnitsky, \textit{On sharp {S}trichartz inequalities in low dimensions.} Int. Math. Res. Not.
\textbf{18} (2006), pages Art. ID 34080.

\bibitem {Kunze:2003:maxi-strichartz-1d} M.~Kunze, \textit{On the existence of a maximizer for the {S}trichartz inequality.} Comm. Math. Phys.
\textbf{243}(1) (2003),137--162.

\bibitem {Shao:2008:maximizers-Strichartz-Sobolev-Strichartz} S.~Shao, \textit{Maximizers for the {S}trichartz inequalities and the
  {S}obolev-{S}trichartz inequalities for the {S}chr\"odinger equation.} 
  Electronic Journal of Differential Equations.\textbf{3} (2009), 1--13.
  
\bibitem{Stein:1993}
E.~M. Stein.
\newblock {\em Harmonic analysis: real-variable methods, orthogonality, and
  oscillatory integrals}, volume~43 of {\em Princeton Mathematical Series}.
\newblock Princeton University Press, Princeton, NJ, 1993.
\newblock With the assistance of Timothy S. Murphy, Monographs in Harmonic
  Analysis, III.

\bibitem {Tao:2006-CBMS-book} T.~Tao, \textit{Nonlinear dispersive equations, Local and global analysis. } volume 106 of  CBMS
  Regional Conference Series in Mathematics Published for the Conference Board of the Mathematical Sciences, Washington, DC, 2006. 

\end{thebibliography}

\end{document}